\documentclass[12pt]{article}

\usepackage[utf8]{inputenc}
\usepackage{graphicx}
\usepackage[all]{xy}
\usepackage{amssymb}
\usepackage{amsmath} 
\usepackage{amsthm}
\usepackage{verbatim}
\usepackage{latexsym}
\usepackage{mathrsfs}

\newenvironment{myItemize}{ 
  \begin{list}{\raisebox{2.2pt}{$\centerdot$}}{%
      \setlength\topsep{5pt}
      \setlength\itemsep{-2pt}p
      \setlength\leftmargin{25pt}
      \setlength\labelwidth{20pt}
    }
  }{
  \end{list}
}

\renewcommand{\labelenumi}{$\mathrm{({\roman{enumi}})}$}

\newenvironment{myEnumerate}{ 
  \begin{list}{\labelenumi}{%
      \usecounter{enumi}
      \setlength\topsep{5pt}
      \setlength\itemsep{-2pt}
      \setlength\leftmargin{25pt}
      \setlength\labelwidth{20pt}
    }
  }{
  \end{list}
}

\newenvironment{myAlphanumerate}{ 
  \begin{list}{\labelenumii}{%
      \usecounter{enumii}
      \setlength\topsep{5pt}
      \setlength\itemsep{-2pt}
      \setlength\leftmargin{25pt}
      \setlength\labelwidth{20pt}
    }
  }{
  \end{list}
}

\makeindex

\renewcommand{\Cup}{\bigcup}

\renewcommand{\a}{\alpha}
\renewcommand{\b}{\beta}
\newcommand{\g}{\gamma}

\renewcommand{\d}{\delta}
\renewcommand{\k}{\kappa}
\renewcommand{\l}{\lambda}
\newcommand{\f}{\varphi}
\renewcommand{\phi}{\varphi}
\renewcommand{\o}{\omega}
\newcommand{\s}{\sigma}

\newcommand{\n}{\eta}

\newcommand{\es}{\varnothing}

\newcommand{\A}{\mathcal{A}}

\newcommand{\Q}{\mathbb{Q}} 
\newcommand{\R}{\mathbb{R}} 
\newcommand{\C}{\mathbb{C}} 

\newcommand{\wins}{\uparrow}

\renewcommand{\P}{\mathbb{P}}             
\newcommand{\forces}{\Vdash}

\newcommand{\ISO}{\operatorname{ISO}}

\newcommand{\DLO}{{\operatorname{DLO}}}
 
\newcommand{\rank}{\operatorname{rank}} 
\renewcommand{\ll}{\mathscr{L}}
 \newcommand{\CUB}{\operatorname{CUB}}
\newcommand{\pr}{\operatorname{pr}}
\newcommand{\rest}{\!\restriction\!}
\newcommand{\dom}{\operatorname{dom}}
\newcommand{\ran}{\operatorname{ran}}
 \renewcommand{\le}{\leqslant}  
 \renewcommand{\ge}{\geqslant}  

 \newcommand{\Sii}{{\Sigma_1^1}}
 
 \newcommand{\Dii}{{\Delta_1^1}}

 \newcommand{\EF}{\operatorname{EF}} 

 \newcommand{\PlOne}{\,{\textrm{I}}}
 \newcommand{\PlTwo}{\textrm{I\hspace{-0.5pt}I}}

\newcommand{\Borel}{\operatorname{Borel}}  
\newcommand{\Baire}{\operatorname{Baire}}  

\swapnumbers
\newtheorem*{Thm*}{Theorem}
\newtheorem{Thm}{Theorem}[section]
\newtheorem{Lemma}[Thm]{Lemma}
\newtheorem{Cor}[Thm]{Corollary}
\newtheorem{Claim}{Claim}[Thm]
\newtheorem{Fact}[Thm]{Fact}

\newenvironment{claim*}{\vspace{7pt}\noindent\textbf{Claim.}}{}

\theoremstyle{definition}
\newtheorem{Def}[Thm]{Definition}

\newtheorem{Open}[Thm]{Open Question}

\theoremstyle{remark}

\newtheorem{RemarkN}[Thm]{Remark}

\newcommand{\proofvpara}{\text{}}

\newenvironment{proofVOf}[1] {\vspace{5pt}\noindent \textit{Proof of #1.}\ignorespaces\renewcommand{\proofvpara}{\text{#1}}}
{\nopagebreak\hspace*{\fill}\mbox{$\square_{\,\proofvpara}$}\\\vspace{-8pt}}

\author{Tapani Hyttinen\footnote{Partially supported by the Academy of Finland through its grant WBS 1251557.}, 
  Vadim Kulikov\footnote{Research supported by the Science Foundation of the University of Helsinki.}\\
 Department of Mathematics and Statistics\\ Gustav H\"allstr\"omin katu 2b\\ 00014, University of Helsinki}
\date{}
\title{Borel* Sets in the Generalized Baire Space}

\usepackage[text={15cm,21cm},centering]{geometry}

\begin{document}

\maketitle

\begin{abstract}
  We start by giving a survey to the theory of $\Borel^{*}(\k)$ sets
  in the generalized Baire space $\Baire(\k)=\k^{\k}$. In particular we look at the relation of
  this complexity class to other complexity classes which we denote by
  $\Borel(\k)$, $\Dii(\k)$ and $\Sii(\k)$ and  
  the connections between $\Borel^*(\k)$ sets and the infinitely deep language $M_{\k^+\k}$.
  In the end of the paper we will prove the consistency of
  $\Borel^{*}(\k)\ne\Sigma^{1}_{1}(\k)$.
\end{abstract}

\vspace{5pt}


\vspace{10pt}


\vspace{10pt}

\vspace{15pt}

Among many classification problems studied in mathematics,
the classification of the subsets of the reals according to their
topological complexity is very classical. It is also very useful:
On one hand
in many branches of mathematics
all subsets of the reals that one really comes across, are of relatively low
complexity. On the other hand
(e.g.) ZFC can prove properties
for these simple sets that it cannot prove for arbitrary
sets. Of many such examples let us mention
the following two: continuum hypothesis is true for Borel sets
(i.e. they are either countable or
of the same size as continuum) while
ZFC does not prove this for arbitrary subsets
of the reals and all $\Sigma^{1}_{1}$
sets are Lebesgue measurable but ZFC also proves the existence of a
non-measurable set. For Borel sets,
see below, and for $\Sigma^{1}_{1}$ sets, see Section \ref{sec:1}.

This classification of the subsets of the reals can also be used
to classify various other mathematical objects. Let
us fix a countable vocabulary $L$. Then every real $r$ can be seen
as a code for an $L$-structure $\A_{r}$ with
the set of natural numbers as the universe
so that every such structure also has a code (not necessarily unique),
see Section 1 for details.
Then one can classify
$L$-theories $T$
(not necessarily first-order)
according
to the complexity of the set $\ISO(T,\o )$ which consist of all
the pairs $(r,q)$ of reals such that
$\A_{r}$ and $\A_{q}$ are isomorphic models of $T$.
This is a much studied classification, but
since this classification captures only countable models of the theories,
it is very different from e.g. the classification of the first-order
theories given by S. Shelah in~\cite{Sh}.

Let $\DLO$ be the theory of dense linear orderings without endpoints.
Then in Shelah's classification $\DLO$ is a very complicated theory
but since $\DLO$ is $\o$-categorical $\ISO(\DLO,\o )$
is very simple, Borel and of very low rank.
On the other hand, M. Koerwien has shown
in \cite{Ko} that there is an $\o$-stable NDOP theory $T$ of depth 2
such that $\ISO(T,\o )$ is not Borel.
In Shelah's classification $\o$-stable NDOP theories of depth 2
are considered very simple.

Beside the general interest in the uncountable, the
considerations like the one above, suggest that it
may make sense to try to generalize
the complexity notions to larger sets. For technical reasons,
the classical theory is usually not developed in the space of real
numbers but
in the Baire space (or Cantor space). The
Baire space is not homeomorphic with
the reals but on the level of Borel sets, it is very close to reals (they are Borel-isomorphic).
And there is a very natural way of generalizing the notion of a Baire space:
Suppose that $\k$ is an infinite cardinal
such that $\k^{<\k}=\k$. Notice that $\o$ satisfies this assumption.
(Some work has been done also without this assumption,
but in general it is not even clear what the right notion of
a Borel set is if this assumption is dropped.)
Now we let the generalized Baire space, $\Baire(\k)$, be the set
of all functions $f\colon\k\rightarrow\k$. We make this into a topological space
by letting the basic open sets be the sets
$N_{\n}=\{ f\in \Baire(\k)\vert\ \n\subseteq f\}$, where
$\n$ is a function from some ordinal $\a <\k$ to $\k$.
From basic open sets we get the family of Borel sets, $\Borel(\k)$,
by closing the collection of all basic open sets under
the unions and intersections
of size $\le\k$. Notice that since $\k^{<\k}=\k$,
open sets are $\Borel(\k)$,
the complements of basic open sets are open and thus by an easy induction
one can also see that the collection of $\Borel(\k)$ sets is closed
under complements. Notice in addition, that if $X$ is a finite product of
the space $\Baire(\k)$ equipped with the product topology, then $X$
is homeomorphic with the space $\Baire(\k)$ giving us
the notion of a $\Borel(\k)$ set also to these spaces.
Alternatively one could let the sets
$N_{(\n_{i})_{i<n}}=\{ (f_{i})_{i<n}\in \Baire(\k)^{n}
\vert\ \n_{i}\subseteq f_{i}\ \forall i<n\}$ be the basic open sets,
where for some $\a <\k$ for all $i<n$, $\n_{i}$ is a function from
$\a$ to $\k$,
and then proceed as in the case of the space $Baire (\k)$.
Of course, $\Baire(\o )$ is the usual Baire space and
$\Borel(\o )$ is the usual family of all Borel sets.

As in the case of reals, for theories $T$ the collections
$\ISO(T,\k)$ can be formed and now for uncountable $\k$,
the classification we get for theories is much closer to that of Shelah's than
in the case $\k =\o$. E.g. for suitable $\k$ and countable
first-order theories $T$,
$T$ is shallow and superstable with NDOP and NOTOP iff
$\ISO(T,\k)$ is $\Borel(\k)$ (in particular,
$\ISO(T,\k)$ is $\Borel(\k)$ for the theory $T$ constructed by Koerwien in \cite{Ko}
and $\ISO(\DLO,\k)$ is not $\Borel(\k)$, in fact not even
$\Delta^{1}_{1}(\k)$, see Section~1).
For more on such questions,
see \cite{FHK} and~\cite{HK}.

In \cite{Bl}, D. Blackwell observed that Borel sets can
be equivalently defined
using games. In \cite{MV}, A. Mekler and J. V\"a\"an\"anen
generalized this game version to get $\Borel^{*}(\k)$ sets. The idea behind
this generalization followed the lines of that of M. Karttunen in \cite{Ka},
where she generalized the logics $L_{\o^{+}\o}$ to logics
$M_{\k^{+}\k}$ via semantic games, see Section \ref{sec:Mkk}
(in turn, this generalization was preceded by a rather similar
generalizations of R. Vaught \cite{Va} and V. Harnik and M. Makkai \cite{HM,Ma} and also
by J. Hintikka and V. Rantala \cite{HR}, see the introduction to this volume
by G. Sandu). The topic of this paper are these $\Borel^{*}(\k)$
sets.
In the first two sections we will give a survey on the
already existing theory and in the
third section we will prove the consistency of $\Borel^{*}(\k)\ne \Sigma^{1}_{1}(\k)$ 
(for uncountable $\k$ such that $\k^{<\k}=\k$),
in fact we will show that it is consistent that $\ISO(\DLO,\k)$
is not $\Borel^{*}(\k)$ (in the first section we will sketch a proof
for the fact that in G\"odel's $L$,
$\Borel^{*}(\k)=\Sigma^{1}_{1}(\k)$ for all uncountable regular $\k$,
and so in $L$, also $\ISO(\DLO,\k)$ is $\Borel^{*}(\k)$).

Acknowledgement: The first author wishes to use this opportunity
to express his gratitude for the guidance and financial support
Jaakko Hintikka gave him during the authors graduate studies.

\section{Borel* Versus Some Other Complexity Classes}\label{sec:1}

Throughout this paper, we assume that $\k$ is an infinite cardinal
and $\k^{<\k}=\k$. Note that from this it follows that $\k$ is regular.

The following definition of $\Borel^{*}(\k)$ sets
is from \cite{Bl} in the case $\k =\o$
and from \cite{MV} in the case $\k$ is uncountable.

\begin{Def}\label{def:Eka}
  Let $\l\le\k$ be a cardinal.
  \begin{myEnumerate}
  \item We say that a tree $T$ is a $\k^{+},\l$-tree
    if does not contain chains of length $\l$ and
    its size is $<\k^+$. We say that it is closed if every chain has
    a unique supremum.
  \item We say that a pair $(T,f)$ is a $\Borel^{*}_{\l}(\k)$-code
    if $T$ is a closed $\k^{+},\l$-tree and $f$ is a function
    with domain $T$ such that if $x\in T$ is a leaf, then
    $f(x)$ is a basic open set and otherwise
    $f(x)\in\{\cup ,\cap\}$.
  \item For an element $\n\in \Baire(\k)$ and
    a $\Borel^{*}_{\l}(\k)$-code $(T,f)$, the $\Borel^{*}$-game
    $B^{*}(\n ,(T,f))$ is played as follows.
    There are two players, $\PlOne$ and $\PlTwo$. The game
    starts from the root of $T$. At each move,
    if the game is at node $x\in T$ and $f(x)=\cap$,
    then $\PlOne$ chooses an immediate successor $y$ of $x$
    and the game continues from this $y$. If $f(x)=\cup$,
    then $\PlTwo$ makes the choice.
    At limits the game continues from the (unique)
    supremum of the previous moves.
    Finally, if $f(x)$ is a basic open set,
    then the game ends, and $\PlTwo$ wins if $\n\in f(x)$.
  \item We say that $X\subseteq \Baire(\k)$ is a $\Borel^{*}_{\l}(\k)$ set
    if it has a $\Borel^{*}_{\l}(\k)$-code
    i.e. that there is a $\Borel^{*}_{\l}(\k)$-code
    $(T,f)$ such that for all $\n\in \Baire(\k)$,
    $\n\in X$ iff $\PlTwo$ has a winning strategy in the game
    $B^{*}(\n ,(T,f))$.
  \item A set is $\Borel^{*}(\k)$ set if it is $\Borel^{*}_{\k}(\k)$ set.
    Similarly we say that $(T,f)$ is a $\Borel^{*}(\k)$-code if it is
    a $\Borel^{*}_{\k}(\k)$-code.
  \item We write $\Borel^{*}(\k)$ also for the family of all $\Borel^{*}(\k)$ sets.
    And we will do the same with the other complexity classes.      
  \end{myEnumerate}
\end{Def}

The observation by D. Blackwell, mentioned in the introduction,
generalizes immediately to the following.

\begin{Lemma}\label{lemma:1.2}
  $\Borel(\k)=\Borel^{*}_{\o}(\k)$.
  In particular, $\Borel(\o )=\Borel^{*}(\o )$ and
  $\Borel(\k)\subseteq \Borel^{*}(\k)$.  
\end{Lemma}
\begin{proof} 
  ``$\subseteq$'': Easy induction on $\Borel(\k)$ sets.
  
  ``$\supseteq$'': Easy induction on the rank of the (well-founded) trees
  in $\Borel^{*}_{\o}(\k)$-codes. 
\end{proof}

\begin{Def}
  \begin{myEnumerate}
  \item $X\subseteq \Baire(\k)$ is $\Sigma^{1}_{1}(\k)$ if it is
    the first projection $\pr_{1}(Y)$ of some closed
    $Y\subseteq \Baire(\k)\times \Baire(\k)$.
  \item $X\subseteq \Baire(\k)$ is $\Delta^{1}_{1}(\k)$ if
    both $X$ and $\Baire(\k)\setminus X$ are $\Sigma^{1}_{1}(\k)$.  
  \end{myEnumerate}
\end{Def}

\begin{RemarkN}\label{remark:Homeo}
  As mentioned in the introduction, for every $n<\o$, the spaces $\Baire(\k)^n$ are homeomorphic to each other, so (i)
  extends to all of them and is equivalent to saying that $X\subset \Baire(\k)^n$
  is $\Sii(\k)$ if it is a projection of a closed $C\subset \Baire(\k)^m$, for some $m>n$.
\end{RemarkN}

In the Lemmas and Theorems below, we show that 
$$\Borel(\k)\subsetneq \Dii(\k)\subseteq \Borel^*(\k)\subseteq \Sii(\k).$$
All these inclusions can be proved in ZFC and the first inclusion is proper.
However it is undecidable in ZFC whether or not the last inclusion is proper 
(Theorems \ref{thm:1.13} and~\ref{thm:DLONotBorelSt}) and it is an open 
problem whether or not it is consistent that the inclusion $\Dii(\k)\subseteq \Borel^*(\k)$
is not proper (Open Question~\ref{que:DiiIsBorel}). 
However in ZFC it can be shown that the inclusion $\Dii(\k)\subsetneq\Sii(\k)$ 
is proper (Lemma~\ref{lemma:1.11}).

\begin{Lemma}\label{lemma:1.4}
  \begin{myEnumerate}
  \item $\Borel(\k)\subseteq\Delta^{1}_{1}(\k)\subseteq\Sigma^{1}_{1}(\k)$.
  \item If $X\subseteq \Baire(\k)$ is
    the first projection of some $\Sigma^{1}_{1}(\k)$ set
    $Y\subseteq \Baire(\k)\times \Baire(\k)$, then it
    is $\Sigma^{1}_{1}(\k)$. In particular, projections
    of $\Borel(\k)$ sets are $\Sigma^{1}_{1}(\k)$  
  \end{myEnumerate}
\end{Lemma}
\begin{proof} (i) The second inclusion is trivial, so we prove the first.
  Since the class of $\Delta^{1}_{1}(\k)$ sets is closed under complements,
  by De Morgan's laws, it is enough to show that
  the class of
  $\Delta^{1}_{1}(\k)$ sets is closed under intersections of size $\le\k$.
  For this, it is enough to show that
  the class of $\Sigma^{1}_{1}(\k)$ sets is closed under
  intersections and unions of size $\le\k$. Let us consider
  intersections first.
  
  Let $C_{i}\subseteq \Baire(\k)\times \Baire(\k)$, $i<\k$, be closed
  sets. It is enough to find a closed set
  $C\subseteq \Baire(\k)\times \Baire(\k)$ such that
  $\pr_{1}(C)=\bigcap_{i<\k}\pr_{1}(C_{i})$.
  Let $X_{i}$, $i<\k$, be a partition of $\k$ into
  sets of size $\k$. For all $j<\k$, by $x^{i}_{j}$ we mean
  the $j$th element of $X_{i}$.
  Then every $\n\in \Baire(\k)$ can be seen as coding
  the unique sequence $(\n_{i})_{i<\k}$ of elements of
  $\Baire(\k)$ so that $\n_{i}(j)=\n (x^{i}_{j})$.
  Now we let $C$ be the set of all $(\xi ,\n )\in \Baire(\k)\times \Baire(\k)$
  such that for all $i<\k$, $(\xi ,\n_{i})\in C_{i}$.
  It is routine to check that $C$ is as wanted.
  
  Now for the unions, let $C_{i}\subseteq \Baire(\k)\times \Baire(\k)$,
  $i<\k$, be again closed
  sets. Now it is enough to find a closed set
  $C\subseteq \Baire(\k)\times \Baire(\k)$ such that
  $\pr_{1}(C)=\bigcup_{i<\k}\pr_{1}(C_{i})$. For every
  $\n\in \Baire(\k)$, let $\n^{-}:\k\rightarrow\k$
  be such that $\n^{-}(\a )=\n (\a )$ if $\a\ge\o$
  and otherwise $\n^{-}(\a )=\n (\a +1)$.
  Now we let $C$ be the set of all
  pairs $(\xi ,\n )$ such that $(\xi ,\n^{-})\in C_{\n (0)}$.
  Clearly $C$ is as wanted.
  
  (ii) Now $Y$ is $\Sii(\k)$, so by Remark \ref{remark:Homeo}
  it is a projection of some closed set $C\subset \Baire(\k)^m$
  and so $X$ is the projection of $C$ as well, so the claim follows by applying Remark \ref{remark:Homeo} again.
\end{proof}

Next we look at the relations between $\Borel^{*}(\k)$ and
other complexity classes. The following theorem
(for $\k >\o$) and especially the clever proof we give, are from~\cite{MV}:

\begin{Thm}\label{thm:MekVaa}
  $\Delta^{1}_{1}(\k)\subseteq \Borel^{*}(\k)$.
\end{Thm}
\begin{proof} 
  Let $A\subseteq \Baire(\k)$ be a $\Delta^{1}_{1}(\k)$ set.
  We need to find a $\Borel^{*}(\k)$-code for it.
  
  Let $C,D\subseteq \Baire(\k)\times \Baire(\k)$ be closed sets
  such that $\pr_{1}(C)=A$ and $\pr_{1}(D)=\Baire(\k)\setminus A$.
  
  For closed $B\subseteq \Baire(\k)\times \Baire(\k)$
  by $T(B)$ we denote the set of all pairs $(\xi\rest\a ,\n\rest\a )$
  such that $(\xi ,\n )\in B$ and $\a <\k$.
  For $\xi\in \Baire(\k)$, by
  $T(\xi ,B)$ we mean the set of all $\n:\a\rightarrow\k$, $\a <\k$,
  such that $(\xi\rest\a ,\n )\in T(B)$ and we order $T(\xi ,B)$ by the subset
  relation. Then $T(\xi ,B)$ is a tree.
  $B$ is closed and therefore we have
  
  \begin{itemize}
  \item[$(*)$] $\xi\in \pr_{1}(B)$ iff $T(\xi ,B)$ contains a branch of length~$\k$.
  \end{itemize}
  Thus, since $\pr_{1}(C)$ and $\pr_{1}(D)$ form a partition of $\Baire(\k)$, we have
  \begin{itemize}
  \item[$(**)$] for all $\xi\in \Baire(\k)$, exactly one of
    $T(\xi ,C)$ and $T(\xi ,D)$ contains a branch of length~$\k$.
  \end{itemize}
  
  For trees $T_{0}$ and $T_{1}$, we write $T_{0}\le T_{1}$
  if there is an order preserving $g\colon T_{0}\rightarrow T_{1}$
  (we do not require that $g$ is one-to-one). 
  Note that $T_{0}\le T_{1}$ iff player $\PlTwo$  has a winning
  strategy in the following game $O(T_{0},T_{1})$:
  At each move $\a$, first $\PlOne$ chooses an element $t_{\a}\in T_{0}$
  and then $\PlTwo$ chooses an element $u_{\a}\in T_{1}$.
  For all $\a<\b$ those elements must satisfy $t_\a<u_\a<t_\b$.
  The player who breaks that rule first, loses.
  
  Now let us look at the tree $T'$ which consists of triples
  $(\xi ,\n ,\d)$ such that $(\xi ,\n )\in T(C)$ and $(\xi ,\d )\in T(D)$.
  The ordering is the obvious one: $(\xi ,\n ,\d )\le (\xi',\n',\d')$
  if $\xi\subseteq\xi'$, $\n\subseteq\n'$ and $\d\subseteq\d'$.
  By $(**)$, $T'$ is a $\k^{+},\k$-tree (in particular, it does not contain
  a branch of length $\k$).
  Now let $T$ be any $\k^{+},\k$-tree such that
  $T''\not\le T'$ (e.g. the tree of all downwards closed chains of $T'$).
  
  Then by $(*)$, for all $\xi\in \Baire(\k)$,
  $T''\le T(\xi ,C)$ iff $\xi\in \pr_{1}(C)$ i.e. iff
  $\PlTwo$ has a winning strategy in $O(T'',T(\xi ,C))$.
  Now it is easy to find a $\Borel^{*}(\k)$-code
  $(T,f)$ such that for all $\xi\in \Baire(\k)$,
  the game $B^{*}(\xi ,(T,f))$ simulates the game
  $O(T'',T(\xi ,C))$. Then $(T,f)$ is
  a $\Borel^{*}(\k)$-code for $A$. 
\end{proof}

\begin{Cor}\label{cor:1.6}
  $\Borel(\o )=\Delta^{1}_{1}(\o )=\Borel^{*}(\o )$. \qed
\end{Cor}

Neither of the identities in Corollary \ref{cor:1.6} above can be proved in
the case $\k >\o$ (at least not in ZFC).
We start with a straightforward one which was
observed in~\cite{FHK}:

\begin{Lemma}
  If $\k >\o$, then $\Borel(\k)\ne\Delta^{1}_{1}(\k)$.  
\end{Lemma}
\begin{proof}
  Recall that by Lemma \ref{lemma:1.2} $\Borel(\k)=\Borel^{*}_{\o}(\k)$.
  Now choose any reasonable coding of $\Borel^{*}_{\o}(\k)$-codes $(t,f)$
  to functions $\n :\k\rightarrow\k$ so that
  if we write $(t_{\n},f_{\n})$ for the pair coded by
  $\n$, every $\Borel^{*}_{\o}(\k)$-code $(t,f)$ is
  $(t_{\n},f_{\n})$ for some $\n$ and the set
  of those $\n$ which code a $\Borel^{*}_{\o}(\k)$-code is closed, or at least $\Borel(\k)$
  (here we need that $\k >\o$, because the property that $t_{\n}$
  is well-founded is not $\Borel(\k)$ if $\k=\o$). Also choose a coding
  for strategies of $\PlOne$ in the games $B^{*}(\xi ,(t,f))$.
  In both cases `almost' any codings works -- excluding the pathological
  ones.
  
  Now, for non-pathological codings,
  it is easy to see that the
  set of all $(\xi ,\n ,\d )$ such that
  $\xi ,\n ,\d\in \Baire(\k)$, $\n$ codes a $\Borel^{*}_{\o}(\k)$-code
  and $\d$ codes a winning strategy of $\PlOne$ in the game
  $B^{*}(\xi ,(t_{\n},f_{\n}))$ is $\Borel(\k)$.
  But then by the Gale-Stewart theorem (and Lemma~\ref{lemma:1.4}~(ii)),
  if we let $A$ be the set of pairs $(\xi ,\n )$ such that
  $\xi ,\n\in \Baire(\k)$, $\n$ codes a $\Borel^{*}_{\o}(\k)$-code
  and $\xi$ is not in the $\Borel(\k)$ set coded by
  $(t_{\n},f_{\n})$, then $A$ is $\Sigma^{1}_{1}(\k)$.
  Similarly one can see that
  the complement of $A$ is also $\Sigma^{1}_{1}(\k)$ and thus
  $A$ is $\Delta^{1}_{1}(\k)$.
  But then also $B=\{\n\in \Baire(\k)\vert\ (\n ,\n )\in A\}$
  is $\Delta^{1}_{1}(\k)$, since $B=\Delta\cap A$ where
  $\Delta =\{ (\n ,\n)\vert\ \n\in \Baire(\k)\}$ is clearly closed
  (and thus $\Delta^{1}_{1}(\k)$, see the proof of Lemma \ref{lemma:1.4}).
  However, $B$ can not be
  $\Borel(\k)$ because obviously it can not
  have a $\Borel^{*}_{\o}(\k)$-code (this is the usual
  Cantor style diagonalization, see the proof of
  Lemma \ref{lemma:1.11} (ii) where we do everything in more detail). 
\end{proof}

However, the other identity, namely $\Delta^{1}_{1}(\k)=\Borel^{*}(\k)$,
is more complicated. In fact,

\begin{Open}\label{que:DiiIsBorel}
   Is $\Delta^{1}_{1}(\k)=\Borel^{*}(\k)$
   together with $\k =\k^{<\k}>\o$ consistent?
\end{Open}

This question is related to Question \ref{que:BorMkk}.
As can be understood from Section~\ref{sec:Mkk},
there is a connection between $\Borel^*(\k)$ sets and
classes of models definable in the language $M_{\k^+\k}$. However the connection is not as close as one might think
e.g. they are different in~$L$, see the discussion after Open Question~\ref{que:BorMkk}.

In any case, one can try to use the intuition
provided by the theory of the language $M_{\k^+\k}$
to understand $\Borel^*(\k)$ sets and this intuition suggests that
the answer to \ref{que:DiiIsBorel} is no:
it seems very unlikely that $\Borel^{*}(\k)$ could be closed
under taking complements (which it would be, if $\Delta^{1}_{1}(\k)=\Borel^{*}(\k)$),
because $M_{\k^+\k}$  is not closed under the negation as is shown
in an unpublished manuscript by T. Huuskonen from the 90's.
But often the proofs from the theory of $M_{\k^{+}\k}$
do not work in the context of $\Borel^{*}(\k)$ and this is also
the case with Huuskonen's proof:
one of the problems in using it
in the context of $\Borel^{*}(\k)$, is
that the models which witness that the sentence does not
have a negation in $M_{\k^{+}\k}$ are necessarily
of size~$>\k$.

The question of the
consistency of $\Delta^{1}_{1}(\k)\ne \Borel^{*}(\k)$ is
easier to handle. In fact,
for every uncountable regular $\k$,
$\Delta^{1}_{1}(\k)\ne \Borel^{*}(\k)$
in $L$ (see Lemma \ref{lemma:1.11}~(ii) and Theorem~\ref{thm:1.13})
and the same holds for an uncountable $\k$ with $\k^{<\k}=\k$ 
also in the model we construct in Section~\ref{sec:2}.
As a preparation, let
us look at the way of seeing this.

\begin{Def}\label{def:1.9}
  \begin{myEnumerate}
  \item We let $\CUB_{\o}(\k)$ be the set of
    all $\n\in \Baire(\k)$ such that the set
    $\{\a <\k\vert\ \n (\a )>0\}$ contains an $\o$-cub set i.e.
    an unbounded set $X\subseteq\k$ which is
    $\o$-closed i.e. if $\a_{i}\in X$ for all $i<\o$, then
    $\cup_{i<\o}\a_{i}\in X$.
  \item A set $X\subseteq \Baire(\n )$ is co-meager if it is
    an intersection of $\k$ many dense and open subsets of $\Baire(\k)$.    
  \item $Y\subseteq \Baire(\k)$ has the property of Baire
    if there are an open set $U$ and a co-meager set $X$ such that
    $Y\cap X=U\cap X$.  
  \end{myEnumerate}
\end{Def}

In \cite{Ha}, A. Halko showed that the classical result that $\Borel(\o)$ sets have
the property of Baire generalizes to 
$\Borel(\k)$ for uncountable $\k=\k^{<\k}$.

The following lemma can be found in \cite{FHK}; item
(iii) was independently known also to P. L\"ucke and P. Schlicht.

\begin{Lemma}\label{lemma:1.10}
  Suppose $\k >\o$.
  \begin{myEnumerate}
  \item $\CUB_{\o}(\k)$ is $\Borel^{*}(\k)$.
  \item $\CUB_{\o}(\k)$ does not have the property of Baire.
  \item It is consistent that every $\Delta^{1}_{1}(\k)$ set has \label{lemma:1.10:itemiii}
    the property of Baire and at the same time $\k =\k^{<\k}>\o$.    
  \item It is consistent that $\k=\k^{<\k}>\o$ and
    $\Delta^{1}_{1}(\k)\ne \Borel^{*}(\k)$.  
  \end{myEnumerate}
\end{Lemma}
\begin{proof} 
  (i) It is easy to see that the set
  $A_{\n}=\{\a <\k\vert\ \n (\a )>0\}$ contains an $\o$-cub set
  iff the player $\PlTwo$ has a winning strategy in
  the game $CG_{\o}(A_{\n})$: the game lasts $\o$ moves.
  At each move $n<\o$, first the player $\PlOne$  chooses an ordinal $\a_{n}\in\k$
  and then $\PlTwo$ chooses an ordinal $\b_{n}\in\k$ such that
  $\b_{n}>\a_{n}$. In the end $\PlTwo$ wins if $\cup_{n<\o}\b_{n}\in A_{\n}$.
  
  But now one
  just needs to find a $\Borel^{*}(\k)$-code $(t,f)$ such that
  the $\Borel^{*}$ game $B^{*}(\n ,(t,f))$ simulates the game
  $CG_{\o}(A_{\n})$. This is easy.

  (ii) Suppose $U$ is open and $X_{i}$, $i<\k$, are open and dense.
  We need to show that $\CUB_{\o}(\k)\cap X\ne U\cap X$
  where $X=\bigcap_{i<\k}X_{i}$. We assume that $U\ne\es$,
  the other case is similar. Now choose an increasing sequence
  $\n_{i}\colon\a_{i}\rightarrow\k$, $\a_{i}<\k$, so that
  \begin{myAlphanumerate}
  \item if $i=0$, then let $\n_{i}$ be such that $N_{\n_{i}}\subseteq U$
    (for $N_{\n_{i}}$, see the introduction),
  \item if $i=j+1$, then let $\n_{i}$ be such that it extends $\n_{j}$
    and $N_{\n_{j}}\subseteq X_{j}$,
  \item if $i$ is limit, then let
    $\n_{i}=(\bigcup_{j<i}\n_{j})\cup\{ (\bigcup_{j<i}\a_{j},0)\}$.    
  \end{myAlphanumerate}
  Now if we let $\n =\cup_{i<\k}\n_{i}$, $\n\in X\cap U$
  but $\n\not\in \CUB_{\o}(\k)$.

  (iii) The statement is forced by adding $\k^+$ many
  Cohen subsets to $\k$,
  for details see~\cite{FHK}.

  (iv) Immediate by (i)-(iii). 
\end{proof}

Let us now turn to the relations between the class $\Sigma^{1}_{1}(\k)$
and the other complexity classes studied above.
The proof of the following lemma is a straightforward generalization from
the case $\k =\o$ and in the case $\k =\o$,
the item (ii) is the famous result of M. Suslin from~\cite{Su}.

\begin{Lemma}\label{lemma:1.11}
  \begin{myEnumerate}
  \item $\Borel^{*}(\k)\subseteq\Sigma^{1}_{1}(\k)$.
  \item $\Delta^{1}_{1}(\k)\ne\Sigma^{1}_{1}(\k)$.
  \end{myEnumerate}  
\end{Lemma}
\begin{proof} (i) Let $(t,f)$ be a $\Borel^{*}(\k)$-code.
  Again one can quite freely choose the way of coding
  strategies of player $\PlTwo$ in the game $B^{*}(\xi ,(t,f))$ to
  functions $\n :\k\rightarrow\k$ and find out that
  the set of those pairs $(\xi ,\n )\in \Baire(\k)\times \Baire(\k)$
  for which $\n$ codes a winning strategy of $\PlTwo$ in the game
  $B^{*}(\xi ,(t,f))$ is closed. And thus the set whose $\Borel^{*}(\k)$-code
  $(t,f)$ is, is $\Sigma^{1}_{1}(\k)$.
  
  (ii) Here we give the easiest proof i.e. we diagonalize, but we will
  return to this question after this proof.
  Let us fix a coding for open sets of $\Baire(\k)\times \Baire(\k)$:
  fix a one-to-one and onto function
  $\pi \colon\k\rightarrow B$, where
  $B$ is the set of all pairs $(f,g)$ functions
  $f,g\colon\a\rightarrow\k$, $\a <\k$. Then we think of $\n\in \Baire(\k)$
  as the code of the open set $U_{\n}=\bigcup_{\a<\k}N_{\n (\a )}$,
  see the alternative way of defining the topology
  on $\Baire(\k)\times \Baire(\k)$ in the introduction.
  Now every open set has a (non-unique) code and
  every $\n\in \Baire(\k)$ codes some open set.
  Now every $\n\in \Baire(\k)$ is also a code for a $\Sigma^{1}_{1}(\k)$ set,
  namely to the set $A_{\n}$ which consists of those
  $\xi\in \Baire(\k)$ such that for some
  $\d\in \Baire(\k)$, $(\xi ,\d )\not\in U_{\n}$.
  Notice that now every $\Sigma^{1}_{1}(\k)$ set has a code.
  
  Now let $A$ be the set of those $\n\in \Baire(\k)$ such that
  $\n\in A_{\n}$. It is easy to see that 
  the set $B=\{ (\n ,\d )\in \Baire(\k)\times \Baire(\k)\vert
  \ (\n ,\d )\not\in U_{\n }\}$ is closed and thus
  $A=\pr_{1}(B)$ is $\Sigma^{1}_{1}(\k)$.
  This set $A$ is not $\Delta^{1}_{1}(\k)$ because if it is,
  then $C=\Baire(\k)\setminus A$ has a code $\n$ which means that
  $\n\in C$ iff $\n\in A_{\n}$ iff $\n\in A$ iff $\n\not\in C$,
  a contradiction. 
\end{proof}

There are also more concrete examples of $\Sigma^{1}_{1}(\k)$ sets
that are not $\Delta^{1}_{1}(\k)$:
Fix a vocabulary $L$ so that it consists of one binary
predicate symbol $\le$ (for simplicity) and
fix also a one-to-one and onto
function $\pi:\k^{2}\rightarrow\k$. Then we let every
$\n\in \Baire(\k)$ code the following $L$-structure $\A_{\n}$:
The universe of $\A_{\n}$ is $\k$ and for all $(x,y)\in \k^{2}$,
the pair $(x,y)$ is in the interpretation of $\le$ if
$\n(\pi (x,y)))\ge 1$. Notice that now every
$L$-structure with universe $\k$ has a code (not unique).
Then, as in the introduction, we let $\ISO(\DLO,\k)$ consists
of those pairs $(\xi ,\n )\in \Baire(\k)\times \Baire(\k)$
such that $\A_{\xi}$ and $\A_{\n}$ are isomorphic
models of the theory $\DLO$. Clearly, $\ISO(\DLO,\k)$
is $\Sigma^{1}_{1}(\k)$. 

By strengthening the methods behind the proof of Theorem~\ref{thm:MekVaa}
and using results from \cite{HT}, it was shown in \cite{MV}, that

\begin{Fact}\label{fact:1.12}
  If $\k >\o$, then $\ISO(\DLO,\k)$ is not $\Delta^{1}_{1}(\k)$.  
\end{Fact}

In fact, this holds for a large
class of first-order theories, see~\cite{FHK,MV}. For more on these questions, see \cite{FHK,HK}.

We finish this section with the following result from~\cite{FHK}:

\begin{Thm}\label{thm:1.13}
  If $V=L$ and $\k >\o$ is regular, then
  $\Borel^{*}(\k)=\Sigma^{1}_{1}(\k)$.  
\end{Thm}
\begin{proof} Let $A\subseteq \Baire(\k)$ be $\Sigma^{1}_{1}(\k)$.
  We need to find a $\Borel^{*}(\k)$-code for it.
  Let $f,g$ be functions with domain $\k$ such that
  \begin{myItemize}
  \item[($\a$)] for all $i<\k$, there is $\g <\k$ such that
    both $f(i)$ and $g(i)$ are
    functions from $\g$
    to $\k$,
  \item[($\b$)] $A$ is the first projection of
    the set
    $$(\Baire(\k)\times \Baire(\k))\setminus\bigcup_{i<\k}N_{(f(i),g(i))}.$$
  \end{myItemize}
  Let $\phi (x,y,z,w,u)$ be the formula of set theory which says that
  \begin{myAlphanumerate}
  \item $x$ and $y$ are functions from $z$ to $z$,
  \item for all $i\in z$,
    either $w(i)$ is not a (proper) subset of
    $y$ or $u(i)$ is not a (proper) subset of $x$
    (i.e. for all $i\in z$ either for all $j\in z$, $(i,y\rest j)\not\in w$
    or for all $j\in z$, $(i,x\rest j)\not\in u$).
  \end{myAlphanumerate}

  Let $\theta =\k^{++}$. Now for all $\xi\in \Baire(\k)$,
  $\xi\in A$ iff $L_{\theta}\models\exists x\phi (x,\xi ,\k ,f,g)$.
  Notice also that $\phi$ is very absolute.
  
  Let $T$ be (e.g.) the theory of $L_{\theta}$ and for all
  $\xi\in \Baire(\k)$, let $C_{\xi}$ be the set
  of all $\a <\k$ such that there is $\b >\a$ with the following properties:
  \begin{myEnumerate}
  \item $\a$ is regular in $L_{\b}$,
  \item $L_{\b}$ is a model of $T$,
  \item $L_{\b}\models\exists x\phi (x,\xi\rest\a ,\a ,f\rest\a ,g\rest\a )$.
  \end{myEnumerate}

  \noindent
  Notice that whether $\a\in C_{\xi}$ or not, depends only on $\xi\rest\a$.

  \begin{Claim}\label{claim:1.13.1}
    For all $\xi\in \Baire(\k)$,
    $\xi\in A$ iff $C_{\xi}$ contains an $\o$-cub set (see Definition \ref{def:1.9}).
  \end{Claim}
  \begin{proofVOf}{Claim \ref{claim:1.13.1}} ``$\Rightarrow$'': Suppose $\xi\in A$.
    For all $\a <\k$, let
    $SH(\a\cup\{ \xi ,\k ,f,g\} )$ be the Skolem closure of
    the set $\a\cup\{ \xi ,\k ,f,g\}$ under the definable
    Skolem functions in $L_{\theta}$ (among the realizations, the Skolem functions choose
    the least one in the definable well-ordering of $L$).
    Let $D$ be the set of those $\a <\k$ such that
    $SH(\a\cup\{ \xi ,\k ,f,g\} )\cap\k=\a$. It is routine to check
    that $D$ contains an $\o$-cub set, in fact it is closed and unbounded.
    But $D\subseteq C_{\xi}$, because if $\a\in D$, then
    the Mostowski collapse of $SH(\a\cup\{ \xi ,\k ,f,g\} )$ is
    $L_{\b}$ for some $\b$ and this $\b$ witnesses that $\a\in C_{\xi}$.
    
    ``$\Leftarrow$'': Suppose $C_{\xi}$ contains an $\o$-cub set $C$.
    For a contradiction, suppose that $\xi\not\in A$ i.e.
    $L_{\theta}\models\neg\exists x\phi (x,\xi ,\k ,f,g)$.
    Following the idea from the above,
    let $D\subseteq\k$ be the set of those
    $\a <\k$ such that $SH(\a\cup\{ \xi ,\k ,f,g,C\} )\cap\k=\a$.
    Again $D$ is closed and unbounded and
    if $\a\in D$ is of cofinality $\o$, then $\a\in C$
    (because $C\cap\a$ is unbounded in $\a$ and $C$ is $\o$-closed).
    
    Let $\a$ be the least limit point
    of $D$. Then $\a\in C\subseteq C_{\xi}$ and $\a\cap D$ has order type $\o$.
    Let $\b^{*}$ be such that $L_{\b^{*}}$ is the Mostowski collapse
    of $SH(\a\cup\{ \xi ,\k ,f,g,C\} )$ and let $\b$ witness
    the fact that $\a\in C_{\xi}$.
    Since 
    $$L_{\b}\models\exists x\phi (x,\xi\rest\a ,\a ,f\rest\a ,g\rest\a )
       \text{ but }L_{\b^{*}}\models\neg\exists x\phi (x,\xi\rest\a ,\a ,f\rest\a ,g\rest\a ),$$
    $\b >\b^{*}$ (the element that witnesses the truth
    of the existential claim can not be in $L_{\b^{*}}$)
    and since $L_{\b}\models T$, $\b$ is also a limit ordinal. Thus since
    $D\cap\a$ is definable in $L_{\b^{*}}$, $D\cap\a\in L_{\b}$.
    Since the order type of $D\cap\a$ is $\o$ and
    $L_{\b}\models T$, it is easy to see that $L_{\b}$
    thinks that $\a$ has cofinality $\o$. This is a contradiction
    since by the definition of $C_{\xi}$, $L_{\b}$
    should think that $\a$ is regular. 
  \end{proofVOf}
  
  Now to find the required $\Borel^{*}(\k)$-code for $A$
  it is enough to find a $\Borel^{*}(\k)$-code $(t,h)$
  such that the game $B^{*}(\xi ,(t,h))$ simulates
  the game $CG_{\o}(C_{\xi})$. This is easy
  (recall that the question of whether
  $\a\in C_{\xi}$ or not depends only on~$\xi\rest\a$). 
\end{proof}

\section{Topological Complexity Classes and $M_{\k^+\k}$}\label{sec:Mkk}

The complexity hierarchy of subsets of $\Baire(\k)$ is reflected by
the definability hierarchy in model theory. 
Fix a coding of models of size $\k$ into elements of $\Baire(\k)$ via
some well-behaved coding $\eta\mapsto \A_\eta$ (for example as the one defined in Section~\ref{sec:1} in connection with
Fact~\ref{fact:1.12}).
We say that $B\subset \Baire(\k)$ is \emph{closed under isomorphism}, if $\eta\in B$
implies $\xi\in B$ for all $\xi$ with $\A_\eta\cong \A_\xi$ and \emph{definable in the logic $L$},
if there exists a sentence $\f\in L$ such that $B=\{\eta\mid \A_\eta\models \f\}$.
Obviously, if $L$ is any reasonable logic and $B$ is definable in $L$, then $B$ is closed under isomorphism.

\begin{Thm}\label{thm:LEV}
  Suppose $B\subset\Baire(\k)$ is closed under isomorphism. Then it is $\Borel(\k)$ if and only if 
  it is definable in $L_{\k^+\k}$.
\end{Thm}

When $\k=\o$, this result is known as the Lopez-Escobar theorem (see e.g. \cite{Ke}) and for $\k=\o_1$ it has been proved by R. Vaught
under CH, see \cite{Va}. Vaught's proof generalizes to any infinite $\k=\k^{<\k}$.

The following definition is due to M. Karttunen \cite{Ka}:
\begin{Def}
  Let $\l$ and $\k$ be cardinals. The language 
  $M_{\l\k}$ is then defined to be the set 
  of pairs $(t,\ll)$ consisting of a closed $\l\k$-tree $t$ (see Definition~\ref{def:Eka}) 
  and a labeling function 
  $$\ll\colon t\to a\cup\{\land,\lor\}\cup \{\exists x_i\mid i<\k\}\cup \{\forall x_i\mid i<\k\}$$ 
  where $a$ is the set of basic formulas, i.e. atomic and negated atomic formulas.
  The labeling $\ll$ satisfies also the following conditions:
  \begin{myEnumerate}
  \item If $x\in t$ is a leaf, then $\ll(t)\in a$.
  \item If $x\in t$ has exactly one immediate successor 
    then $\ll(t)$ is either $\exists x_i$ or $\forall x_i$ for some $i<\k$.
  \item Otherwise $\ll(t)\in\{\lor,\land\}$.
  \item If $x<y$, $\ll(x)\in \{\exists x_i,\forall x_i\}$ and $\ll(y)\in \{\exists x_j,\forall x_j\}$, then $i\ne j$.
  \end{myEnumerate}

  The truth of $M_{\l\k}$ is defined in terms of a semantic game. Let $(t,\ll)$ be a
  sentence and let $\A$ be a model. In the semantic game $S(\f,\A)=S(t,\ll,\A)$ for $M_{\l\k}$ 
  the players start at the root of $t$ and climb up one step at a time. 
  Suppose that they are at the element $x\in t$. If $\ll(x)=\lor$, then player $\PlTwo$ chooses an immediate successor of $x$,
  if $\ll(x)=\land$, then player $\PlOne$ chooses an immediate successor of $x$.
  If $\ll(x)=\forall x_i$ then player $\PlOne$ picks an element $a_i\in\A$ and if $\ll(x)=\exists x_i$ then player $\PlTwo$ picks $a_i\in \A$
  and they move to the immediate successor of $x$. If they come to a limit, they move to the unique supremum. If $x$ is
  a maximal element of $t$, then they plug the elements $a_i$ in place of the corresponding free variables in the basic formula
  $\ll(x)$ and if the resulting sentence is true, then player $\PlTwo$ wins.
  $\A\models (t,\ll)$ if and only if $\PlTwo$ has a winning strategy in the semantic game.
\end{Def}

One immediately sees some similarity with the definition of the $\Borel^*(\k)$ sets and that maybe there is some 
hope to prove a result similar to Theorem~\ref{thm:LEV}. Employing this intuition, 
the following was shown in \cite{FHK} (the key idea is due to S. Coskey and P. Schlicht):

\begin{Thm}\label{thm:MkkBor}
  If $B\subset \Baire(\k)$ is $\Borel^*(\k)$ and closed under isomorphism, then
  it is definable in $\Sii(M_{\k^+\k})$.
\end{Thm}

The converse of \ref{thm:MkkBor} is consistent:

\begin{Thm}[$V=L$]\label{thm:MkkinL} 
  Let $\k>\o$ be regular.
  If $B\subset\Baire(\k)$ is definable in $\Sii(M_{\k^+\k})$, then $B$ is $\Borel^*(\k)$.
\end{Thm}
\begin{proof}
  By Theorem \ref{thm:1.13}, if $B$ is $\Sii(\k)$, then it is $\Borel^*(\k)$, so we have to show
  that $B$ is $\Sii(\k)$ whenever it is definable in $\Sii(M_{\k^+\k})$. 
  But if $B$ is definable by a formula $\exists R\f(R)$ where $\f$ is in $M_{\k^+\k}$ and $R$
  is a second order variable, then $B$ is the projection of a set definable in $M_{\k^+\k}$
  via the formula $\f$ in the vocabulary extended by $\{R\}$. Thus the result follows from 
  Theorem \ref{thm:DefMisBor} below and Lemma~\ref{lemma:1.11}.
\end{proof}

\begin{Thm}\label{thm:DefMisBor}
  If $B\subset\Baire(\k)$ is definable in $M_{\k^+\k}$, then it is $\Borel^*(\k)$.
\end{Thm}
\begin{proof}
  Given a sequence $\bar a=(a_0,\dots,a_n)$ of $\k$ and a basic formula $\f(\bar a)$, let $N(\f(\bar a))$
  be the set of all $\eta$ such that $\A_\eta\models \f(\bar a)$. Clearly $N(\f(\bar a))$ is an open set.

  Let $t$ be a tree and $\ll$ a labeling function such that $(t,\ll)$ is a sentence in $M_{\k^+\k}$. 
  Let $t^*$ consist of functions $f$ such that $\dom f$ is a downward closed linear sub-order of $t$ with a maximal element,
  and $\ran f$ is $\k$ and
  if $x\in\dom f$, but $\ll(x)\notin \{\exists x_i\mid i<\k\}\cup \{\forall x_i\mid i<\k\}$,
  then $f(x)=0$. Order $t^*$ by $f<_{t^*}g\iff f\subset g$.
  If $f$ is a leaf of $t^*$, then $\dom f$ is a branch and there is a maximal element $x\in \dom f$ which is also a maximal element
  in $t$.
  Let $A=\{i<\k\mid \exists y\in\dom f (\ll(y))\in \{\exists x_i,\forall x_i\}\}$
  Then for each $i\in A$, let $\a_i$ be the ordinal such that $f(y)=\a_i$ where $y$ is the unique element of
  $\dom f$ such that $\ll(y) \in \{\exists x_i,\forall x_i\}$.
  Then let $h(f)=N(\ll(x)((\a_{i})_{i\in A}))$, where $\f((\a_{i})_{i\in A}))$ is the sentence obtained
  from the formula $\f$ by replacing the free variable $x_i$ with $\a_i$ whenever $x_i$ occurs (if ever).
  Note that this $h(f)$ is not necessarily a basic open set, but note that in the definition 
  of $\Borel^*(\k)$ sets, basic open sets can be replaced by any open sets (even any Borel sets) and obtain an equivalent definition.
  If $\max\dom f$ is not a leaf, then let $h(f)=\cup$, if $\ll(\max\dom f)\in\{\lor\}\cup\{\exists x_i\mid i<\k\}$
  and $h(f)=\cap$ otherwise. Then $(t^*,h)$ is a $\Borel^*(\k)$-code for the set defined by $(t,\ll)$.
\end{proof}

A \emph{dual} of a formula of $M_{\k^+\k}$ is obtained by switching all conjunctions to disjunctions,
existential quantifiers to universal quantifiers and vice versa and the basic formulas to their
first-order negations. A formula is \emph{determined} if either the formula or its dual holds in every model.
In a similar way define a dual of a $\Borel^*(\k)$ set and determined $\Borel^*(\k)$ set.
Applying a separation theorem of \cite{MV} that every disjoint $\Sii(\k)$ sets can be separated by
a $\Borel^*(\k)$ set and its dual (a stronger version of Theorem \ref{thm:MekVaa} above)
and a separation theorem of H. 
Tuuri \cite{Tu} which says that every two inconsistent 
$\Sii(M_{\k^+\k})$-sentences can be separated by an $M_{\k^+\k}$-sentence and its dual, we have a corollary:

\begin{Cor}
  The following are equivalent for a set $D\subset \Baire(\k)$:
  \begin{myItemize}
  \item  $D\subset \Baire(\k)$ is $\Dii(\k)$ and closed under isomorphism,
  \item  both $D$ and $\Baire(\k)\setminus D$ are definable in $M_{\k^+\k}$,
  \item  $D$ is definable by a determined $M_{\k^+\k}$-formula,
  \item  $D$ is a determined $\Borel^*(\k)$ set.
  \end{myItemize}
\end{Cor}

However, the converse of \ref{thm:DefMisBor} is not known to be consistent:

\begin{Open}\label{que:BorMkk}
  Is it consistent that the sets $B\subset\Baire(\k)$ definable in $M_{\k^+\k}$ are precisely 
  the $\Borel^*(\k)$ sets closed under isomorphism?
\end{Open}

The negation holds in $L$ by Theorem \ref{thm:MkkinL}, because
provably there is a $\Sigma^{1}_{1}(L_{\o\o})$-sentence
which expresses a property which is not
expressible in $M_{\k^{+}\k}$,
not even on models of size $\k$.
(The property is the following:
the models consist of two distinct
linear orderings and the sentence says that
the linear orderings are isomorphic.)

At least one source of difficulty here seems to be the following difference between the definitions of $\Borel^*(\k)$-codes
and $M_{\k^+\k}$-sentences: in a $\Borel^*(\k)$-code $(t,h)$, the attachment $h$ of open sets to the leaves,
can be completely arbitrary, but in a $M_{\k^+\k}$-sentence $(t,\ll)$, the truth value of the basic formula
$\ll(x)$, for a leave $x$, depends in a continuous way on the moves that the players have chosen during the game (namely
which interpretations they have chosen for the quantifiers).

\section{Consistency of $\Borel^*(\k)\ne\Sii(\k)$}\label{sec:2}

\begin{Thm}[ZFC]\label{thm:DLONotBorelSt}
  It is consistent 
  that $\ISO(\DLO,\k)$ is not $\Borel^*(\k)$ and at 
  the same time $\Dii(\k)\subsetneq \Borel^*(\k)$ and $\k^{<\k}=\k$.
\end{Thm}
\begin{proof}
  We start from a model in which 
  $\k^+=2^\k$ and $\k^{<\k}=\k>\o$
  (for instance from $L$) and force the 
  statement with a $<\k$-closed, $\k^+$-c.c. forcing.
  Given a code $(t,h)$ of a $\Borel^*(\k)$ subset of $\Baire(\k)\times \Baire(\k)$,
  we will design a forcing p.o. $\R(t,h)$ such that $\R(t,h)\forces B(\check t,\check h)\ne \ISO(\DLO,\k)$, where $B(t,h)$
  is the $\Borel^*(\k)$ set coded by $(t,h)$. 
  By iterating this forcing we shall kill all possible $\Borel^*(\k)$-code candidates for $\ISO(\DLO,\k)$.
  By combining this forcing with the Cohen forcing $2^{<\k}$, we will be able to
  show, using methods from \cite{FHK}, that in the generic extension also $\Dii(\k)\subsetneq \Borel^*(\k)$.
  
  Given trees $t,t^*$, let us define the game $H(t,t^*)$. At the $\g$:th move, player $\PlOne$ picks
  a pair $(a_\g,b_\g)\in t\times t^*$
  and then player $\PlTwo$ picks an element $c_{\g}\in t^*$. The rules declare the following. 
  If $\g<\g'$, then we must have $b_\g<c_{\g}<b_{\g'}$
  and $a_{\g}<a_{\g'}$. The first player who breaks the rules has lost the game.
  
  We will first find for each $\k^+\k$-tree $t$ a 
  $<\k$-closed $\k^+$-c.c. forcing $\P(t)$ 
  such that $\P(t)\forces\exists t^*(\PlTwo\uparrow H(\check{t},t^*))$. 
  The order $\P(t)$ will 
  consist of triples $(P,U,f)$, where intuitively, $P$ approximates $t^*$, $U$ cuts the branches of $t^*$
  and $f$ approximates the winning strategy of $\PlTwo$ in $H(t,t^*)$. We require $(P,U,f)$ to satisfy the following:
  \begin{itemize}
  \item[P1] $P\subset \k^{<\k}$ is closed downward, 
  \item[P2] $U\subset \k^{<\k}$ is an antichain,
  \item[P3] If $q\in U$, then $\dom q$ is a limit ordinal and $\forall p\in P(p\not\supset q)$,
  \item[P4] $f$ is a function with $\dom f\subset(t\times P)^{<\a}$ for some $\a<\k$ and $\ran f\subset P$,
  \item[P5] If $p=((a_{i},b_{i}))_{i<\b}\in \dom f$, then $p$ is strictly increasing in the coordinatewise ordering of $t\times P$
    and $b_i<f((a_i,b_i)_{i<\b})$ for all $i<\b$.
  \item[P6] If $p,q\in \dom f$, $\dom p=\dom q=\a+1$, $p\ne q$ and $p\rest\a=q\rest \a$, then $f(p)$ and $f(q)$ are incomparable.
  \item[P7] If $p\rest\b\in \dom f$ for some $p\in (t\times P)^{<\k}$ and all $\b<\a=\dom p$, then 
    \mbox{$\displaystyle\Cup_{\b<\a}f(p\rest\b)\notin U$}.
  \end{itemize}
  The order on $\P(t)$ we define as follows: $(P,U,f)<(P',U',f')$, if
  \begin{itemize}
  \item[O1] $P\subset P'$, $U\subset U'$ and $f\subset f'$,
  \item[O2] if $p\in \dom f'\setminus \dom f$, then $f'(p)>\a$, where $\a$
    is the smallest ordinal such that $P\cup U\cup\ran f\subset \a^{<\a}$. Call this $\a$
    the \emph{rank} of $(P,U,f)$ and denote $\a=\rank(P,U,f)$.
  \end{itemize}
  
  Next we show that $\P(t)$ is as wanted. 
  
  \begin{Claim}\label{claim:C1}
    $\P(t)$ is $<\k$-closed.
  \end{Claim}
  \begin{proofVOf}{Claim \ref{claim:C1}}
    Suppose $(p_\b)_{\b<\a}$, $p_\b=(P_\b,U_\b,f_\b)$,
    is an increasing sequence of conditions
    of limit length $\a<\k$. Then let 
    $$p_\a=(P_\a,U_\a,f_\a)=(\Cup_{\b<\a}P_\b,\Cup_{\b<\a}U_\b,\Cup_{\b<\a}f_\b)$$
    and let us show that $p_\a\in \P(t)$ and $p_{\a}>p_\b$ for all $\b<\a$.
    To check that $p_\a\in \P(t)$, note that all conditions except P7 are local and easy to check.
    For the condition P7, suppose that $p\rest\b\in\dom f_\a$ for all $\b<\dom p$ and assume for a contradiction
    that $\Cup_{\b<\dom p}f_\a(p\rest\b)\in U_\a$. But then $\Cup_{\b<\dom p}f_\a(p\rest\b)\in U_\g$ for some $\g<\a$.
    This means by O2, that the values of $f_{\g+1}$ are above $f_\a(p\rest\b)$ for all $\b<\dom p$ which is a 
    contradiction unless $\Cup_{\b<\dom p}f_\a(p\rest\b)=\Cup_{\b<\dom p}f_\g(p\rest\b)$. But the latter is
    a contradiction with P7 applied to~$p_{\g}$.
  \end{proofVOf}
  
  Let $G$ be $\P(t)$-generic and let
  $$t^*=\Cup\{P\mid (P,U,f)\in G\text{ for some }U,f\}.$$  
  \begin{Claim}\label{claim:C2}
    In the $\P(t)$-generic extension $t^*$ is a $\k^+\k$-tree. 
  \end{Claim}
  \begin{proofVOf}{Claim \ref{claim:C2}}
    We must show that there are no branches of length $\k$.
    Suppose on contrary that $b$ is a branch and let $\dot b$ be the $\P(t)$-name for $b$.
    Suppose $p_0=(P_0,U_0,f_0)$ forces that $\dot b$ is a branch and suppose $(P_1,U_1,f_1)=p_1>p_0$.
    By induction define $p_{\a+1}=(P_{\a+1},U_{\a+1},f_{\a+1})$ assuming that $p_\a=(P_{\a},U_\a,f_\a)$
    is already defined, such that $p_{\a+1}$ decides $\dot b$ up to $\rank(p_{\a})$. Suppose $\a$ is a limit
    and that $p_\b$ has been defined for $\b<\a$ and for every $\b<\a$, $p_{\b+1}$ has evaluated 
    $\dot b$ up to $\b$, from which it follows that it has been evaluated up to $\a$ in fact.
    Denote this evaluated branch by $e_\a$.
    If $\cup e_{\a}\subset \ran f$, then just continue: let $p_\a=\sup_{\b<\a}p_\b$ which is well defined
    by Claim~\ref{claim:C1}. Otherwise let $U_\a=\Cup_{\b<\a}U_\b\cup \{\dot b\rest\a\}$,
    $f_\a=\Cup_{\b<\a}f_\b$ and $P_\a=\Cup_{\b<\a}P_\b$: then $p_\a=(P_\a,U_\a,f_\a)$ marks an end to the branch
    $\dot b\rest\a$ which is a contradiction, because $p_\a>p_\b$ for $\b<\a$ (P7 is satisfied, 
    because $\cup e_{\a}\not\subset \ran f$). 
    So we need to show that this process terminates,
    i.e. the ``otherwise''-part of the previous sentence is satisfied at some point.
    If it does not terminate, then we obtain a branch in $\ran f$, but $f$ 
    is a strategy in the game and by the property P6, this branch determines a branch in~$t$
    which is a contradiction, because $t$ is $\k^+\k$-tree.
  \end{proofVOf}
  
  Let $G$ be $\P(t)$-generic and let
  $$g=\Cup\{f\mid (P,U,f)\in G\text{ for some }P,U\}.$$
  \begin{Claim}\label{claim:C3}
    In the $\P(t)$-generic extension, $g$ is a winning strategy of player $\PlTwo$ in $H(t,t^*)$.
  \end{Claim}
  \begin{proofVOf}{Claim \ref{claim:C3}}
    If $s$ is a strategy of $\PlOne$, let $\dot s$ be a name for $s$ and let $\dot g$ be a
    name for $g$. We will show that $\P(t)$ forces that $\dot g$ beats $\dot s$.
    It is enough to show that $\PlTwo$ can always follow the rules, so suppose
    they have played $\a$ moves and suppose that $p\in\P(t)$ decides the game $s*g$ (the
    game in which those strategies are used) up to the move $\a$. Find a $q>p$ which decides
    the next move given by $g$. By definition of $\P(t)$ this will 
    follow the rules. Essential here is that since $\P(t)$ is closed, every play of length $<\k$ is already
    in the ground model.
  \end{proofVOf}
  
  \begin{Claim}\label{claim:C4}
    Denote by $\dot t^*$ a $\P(t)$-name for $t^*$ defined by 
    $\dot t^*=\{(\check p,q)\mid q\in \P(t), q=(P,U,f)\text{ and }p\in P\}$.
    The forcing $\P(t)*\dot t^*$ contains a dense sub-order $\R$ which is $<\k$-closed. 
  \end{Claim}
  \begin{proofVOf}{Claim \ref{claim:C4}}
    By definition $(q,\rho)\le (q',\rho')$, if $q\le q'$ and $q'\forces \rho\le \rho'$.
    It is easy to see that the suborder $\R'$ of $\P(t)*\dot t^*$ consisting of the pairs $(q,\check p)$ such that 
    $(\check p,q)\in \dot t^*$ is dense.
    Let $\R$ be the subset of $\R'$ consisting of those $(q,\check p)$ for which
    $\dom(p)\ge \sup \{\dom \eta\mid \eta\in U_q\}$ where $q=(P_q,U_q,f_q)$ $\mathbf{(*)}$.
    It is again easy to see that $\R$ is dense. 

    Suppose $(q_i,\check p_i)_{i<\a}$ is an increasing sequence in $\R$ of length $\a<\k$. 
    Let $q_\a=\sup_{i<\a}q_i$ in $\P(t)$ and $p_\a=\Cup_{i<\a}p_i$.
    Then $q_\a$ is of the form $(P,U,f)$ and by $\mathbf{(*)}$ it is possible to extend $P$ to $P'$ such that $p_\a\in P'$
    and $q'_\a=(P',U,f)$ is still in $\P(t)$. But then $(q'_\a,\check p_\a)\in\R$.
  \end{proofVOf}
  
  \begin{Claim}\label{claim:C5}
    For each $(t,h)$ there exists a $\k^+$-c.c. $<\k$-closed forcing $\R(t,h)$
    such that in the $\R$-generic extension 
    $\ISO(\DLO,\k)$ is \emph{not}
    the $\Borel^*(\k)$ set coded by $(t,h)$. 
  \end{Claim}
  \begin{proofVOf}{Claim \ref{claim:C5}}
    If $\P(t)$ forces that, let $\R(t,h)=\P(t)$. Otherwise let $\R(t,h)$ be the dense sub-order
    of $\P(t)*\dot t^*$ given by Claim~\ref{claim:C4}. Let us 
    show that this works. It is sufficient to show that $\P(t)*\dot t^*$
    forces the statement. Let us work in the $\P(t)$-generic extension $V[G]$.
    Let $\eta,\xi\in 2^\k$ be such that
    $\A_\eta$ and $\A_\xi$ are non-isomorphic models of $\DLO$, but $\PlTwo\wins \EF_{t^*}(\A_\eta,\A_\xi)$.
    These can be found by \cite{HT}.
    Since $\P(t)$ didn't force the statement, the pair $(\eta,\xi)$ is not in the set coded by $(t,h)$.
    Now forcing with $t^*$ adds a branch to $t^*$ and since $t^*$ can be embedded into the tree of partial isomorphisms
    between $\A_\eta$ and $\A_\xi$ via the winning strategy of $\PlTwo$ in $\EF_{t^*}(\A_\eta,\A_\xi)$, it adds a branch also
    to that tree, and so $\A_\eta$ and $\A_\xi$ are isomorphic in $V[G][G_0]$, where $G$ is $\P(t)$-generic over $V$
    and $G_0$ is $t^*$-generic over $V[G]$. Next we show, that in $V[G][G_0]$, $(\eta,\xi)$ is
    not in the $\Borel^*(\k)$ set coded by $(t,h)$. 
    
    On contrary, assume that $V[G][G_0]\models (\eta,\xi)\in B(t,h)$ and let us show that
    then $V[G]\models (\eta,\xi)\in B(t,h)$, which is a contradiction. Let $\sigma$ be a winning strategy
    of player $\PlTwo$ in $V[G][G_0]$ in $B^*((\eta,\xi),(t,h)))$, as in the definition of $\Borel^*(\k)$,
    and let $\dot \s$ be a name for $\s$. Let us show how $\PlTwo$ has to play to win $B^*((\eta,\xi),(t,h))$
    in $V[G]$. For that, let $g$ be a winning strategy of player $\PlTwo$ in $H_t(t^*)$
    which exists in $V[G]$ by Claim~\ref{claim:C3}. 
    
    Assume that $a_0$ is the first move of $\PlOne$ in $B^*((\eta,\xi),(t,h))$.
    Player $\PlTwo$ finds a condition $c_0$ in $t^*$ which decides $\dot \sigma$ far enough to give an
    answer $b_0$ to that move. Player $\PlTwo$ answers in $B^*((\eta,\xi),(t,h))$ with $b_0$ and at the 
    same time imagines that $(b_0,c_0)$ is the first move of $\PlOne$ in $H_t(t^*)$ and replies using
    $g$ in this imaginary game by $d_0>c_0$.
    Suppose that the players have played $(a_i,b_i)_{i<\a}$ in $B^*((\eta,\xi),(t,h))$
    so that $a_i$ are the moves of player $\PlOne$ and $b_i$ are the moves of player $\PlTwo$.
    At the same time player $\PlTwo$ has constructed a sequence $(c_i,d_i)_{i<\a}$ using the
    imaginary game. Next player $\PlOne$ picks $a_\a$ in $B^*((\eta,\xi),(t,h))$. Player $\PlTwo$
    solves $\dot\s$ by a condition $c_\a>\sup_{\b<\a}d_\b$ so that she obtains an answer $b_\a$
    and again imagines that $(b_\a,c_\a)$ is just the next move of $\PlOne$ in $H_t(t^*)$ and
    picks $d_\a$ using $g$. In this way the players will climb up a branch $b\subset t$ 
    with the basic open set $h(b)$ in the end. By definition $h(b)=N_p$ for some $p\in 2^{<\k}$
    in $V$, and neither $\P$ nor $\P*t^*$ adds small subsets 
    (Claims \ref{claim:C1} and~\ref{claim:C4}), 
    so $h(b)^V=h(b)^{V[G]}=h(b)^{V[G][G_0]}$.
    Now since $\sigma$ was winning in $V[G][G_0]$, the above described strategy is winning in $V[G]$.
  \end{proofVOf}
  
  Thus, for a code $(t,h)$ we have constructed a forcing $\R(t,h)$ which forces that
  $$\ISO(\DLO,\k)\ne B(t,h).$$ Using this fact, we will define a $<\k$-support iterated forcing $\Q$ of length $\k^+$
  such that in the $\Q$-generic extension
  there are no pairs $(t,h)$ such that $\ISO(\DLO,\k)=B(t,h)$ \emph{at all} which means that $\ISO(\DLO,\k)$ is not
  $\Borel^*(\k)$ and moreover $\Q\forces\Dii(\k)\subsetneq \Borel^*(\k)$
  
  Let $s\colon \k^+\to\k^+\times\k^+$ be onto such that $s_2(\a)<\a$ for $\a<\k^+$.
  Define the $<\k$-support iterated forcing construction (see \cite[Ch. VIII]{Ku})
  $$(\P_\b,\rho_\b)_{\b<\k^+}\text{ along with a sequence }\s(\a,\b)$$
  as follows. 
  For each $\b<\k^+$, let $\{\s(\a,\b)\mid \a<\k^+\}$ be the enumeration of all $\P_\b$-names for codes for $\Borel^*(\k)$ sets 
  and $\rho_\b$ is a $\P_\b$-name for the Cohen forcing $\C=2^{<\k}$, if $\b$ is odd (of the form $\a+2n+1$ with
  $\a$ a limit and $n<\o$) and 
  $\rho_\b$ is a $\P_\b$-name for $\R(\dot t,\dot h)$ with $(\dot t,\dot h)=\s(s(\b))$, if $\b$ is even.

  It is easily seen that $\P_\g$ is $<\k$-closed and has the $\k^+$-c.c. for all $\g\le \k^+$.  
  We claim that $\Q=\P_{\k^+}$ forces that $\ISO(\DLO,\k)$ is not $\Borel^*(\k)$.
  Let $G$ be $\P_{\k^+}$-generic and let $G_\g=\text{``}\,G\cap \P_\g\text{''}$ for every $\g<\k$.
  Then $G_\g$ is $\P_{\g}$-generic.
  
  Suppose that in $V[G]$, $\ISO(\DLO,\k)=B(t,h)$ for some $(t,h)$.
  By \cite[Theorem VIII.5.14]{Ku}, there is $\d<\k^+$ such that
  $(t,h)\in V[G_\d]$. Let $\d_0$ be the smallest such~$\d$.
  
  Now there exists $\sigma(\g,\d_0)$, a
  $\P_{\d_0}$-name for $(t,h)$. By the definition of $s$, there exists an even $\d>\d_0$ with $s(\d)=(\g,\d_0)$.
  Thus 
  $$\P_{\d+1}\forces \text{``}\sigma(\g,\d_0)\text{ is not a $\Borel^*(\k)$-code for $\ISO(\DLO,\k)$''},$$
  i.e. $V[G_{\d+1}]\models B(t,h)\ne \ISO(\DLO,\k)$.
  We want to show that this holds also in $V[G]$. 
  In $V[G_{\d+1}]$ define 
  $$\P^{\d+1}=\{(p_i)_{i<\k^+}\in \P_{\k^+}\mid (p_i)_{i<\d+1}\in G_{\d+1}\}.$$
  Then $\P^{\d+1}$ has $\k^+$-c.c. and is $<\k$-closed because at each stage of the iteration the
  forcings have these properties and the iteration has $<\k$-support. 
  Assume that $G^{\d+1}$ is $\P^{\d+1}$-generic over $V[G_{\d+1}]$.
  We will show that in $V[G_{\d+1}][G^{\d+1}]$ we have $B(t,h)\ne \ISO(\DLO,\k)$. On the other hand
  $V[G]=V[G_{\d+1}][G^{\d+1}]$ for some $G^{\d+1}$, so this finishes the proof of the part of the theorem
  concerning $\ISO(\DLO,\k)$.
  
  There are two cases. First assume that there are
  $\eta$ and $\xi$ in $V[G_{\d+1}]$ such that $\A_\eta$ and $\A_\xi$ are isomorphic linear orders
  and $V[G_{\d+1}]\models (\eta,\xi)\notin B(t,h)$. 
  Then in $V[G_{\d+1}][G^{\d+1}]$, we have still that $\A_\eta$ and $\A_\xi$ 
  are isomorphic, but $(\eta,\xi)\notin B(t,h)$: $\P^{\d+1}$ does not
  add small sets and it does not add a winning strategy of $\PlTwo$ in the game
  $B^*((\eta,\xi),(t,h))$, because otherwise we 
  would obtain a winning strategy already in $V[G_{\d+1}]$ using $<\k$-closedness
  of $\P^{\d+1}$ in an argument similar to the one in the end of the proof of Claim~\ref{claim:C5}. 
  
  The other case is that there are
  $\eta$ and $\xi$ in $V[G_{\d+1}]$ such that $\A_\eta$ and $\A_\xi$ are non-isomorphic linear orders
  and $V[G_{\d+1}]\models (\eta,\xi)\in B(t,h)$. Now dually to the first case, the winning strategy
  of $\PlTwo$ in $B^*((\eta,\xi),(t,h))$ remains a winning strategy, because otherwise we would be able
  to beat it already in $V[G_{\d+1}]$ using the closedness of $\P^{\d+1}$. On the other hand
  $\A_\eta$ and $\A_\xi$ do not become isomorphic, because that would add a winning strategy of $\PlTwo$
  in $\EF_\k(\A_{\eta},\A_{\xi})$ which is impossible by the same argument.
  
  Now we are left to show that $\Dii(\k)\ne\Borel^*(\k)$ in the generic extension by $\Q$.
  The $\k^+$-long $<\k$-support
  iteration of the Cohen forcing $\C$ yields a model in which $\Dii(\k)$ sets have the property of Baire
  and the same proof works in this case, because in our iteration every other step was~$\C$.
  But this in turn implies that $\Dii(\k)\subsetneq \Borel^*(\k)$, see Lemma \ref{lemma:1.10} above.
\end{proof}

The following answers a question asked in \cite{FHK}:

\begin{Cor}\label{cor:DiiBoSii}
  It is consistent that $\Dii(\k)\subsetneq \Borel^*(\k)\subsetneq \Sii(\k)$ and $\k^{<\k}=\k>\o$.
\end{Cor}
\begin{proof}
  $\ISO(\DLO,\k)$ is $\Sii(\k)$, so the result follows from Theorem~\ref{thm:DLONotBorelSt}.
\end{proof}

\bibliography{ref}{}
\bibliographystyle{alpha}
\end{document}